\documentclass[a4paper, 11pt,reqno]{amsart}

\usepackage{amsmath,amsthm,amsbsy,amssymb}
\usepackage{arydshln} % dashline in table
\usepackage{booktabs} %\toprule, \midrule \addlinespace
\usepackage{paralist} % special list compactenum

\makeatletter
\newcommand{\leqnos}{\tagsleft@true\let\veqno\@@leqno}
\newcommand{\reqnos}{\tagsleft@false\let\veqno\@@eqno}
\reqnos
\makeatother
 \usepackage[backref=page]{hyperref}
\hypersetup{
    colorlinks = true,
linkcolor={red},
urlcolor={blue},
citecolor={green},    
urlcolor = {blue},
citebordercolor = {0.33 .58 0.33},
 linkbordercolor = {0.99 .28 0.23}
}

\usepackage{mathrsfs}

\newcommand{\cS}{\mathcal S}
\newcommand{\cT}{\mathcal T}

\newcommand{\fM}{\mathfrak M}

\usepackage{bm}
\newcommand*{\B}[1]{\ifmmode\bm{#1}\else\textbf{#1}\fi}
\newcommand{\bv}{\B{v}}
\newcommand{\bvi}{\B{v_1}}
\newcommand{\bvii}{\B{v_2}}
\newcommand{\bw}{\B{w}}
\newcommand{\bc}{\B{c}}

\newcommand{\ba}{\B{a}}

\newcommand{\cV}{\mathcal V}
\newcommand{\cW}{\mathcal W}

\newcommand{\RR}{\mathbb{R}}

\newcommand{\sdfrac}[2]{\mbox{\small$\displaystyle\frac{#1}{#2}$}}

\newcommand{\ltfrac}[2]{\mbox{\large$\frac{#1}{#2}$}}

\newcommand{\norm}[1]{\left|\hspace*{2.5pt} \!\!\left| #1\right|\hspace*{2.5pt} \!\!\right|}
\newcommand{\normu}[1]{\pmb{\left\vert\vphantom{#1}\right.}#1\pmb{\left.\vphantom{#1}\right\vert}}

\newcommand{\abs}[1]{\left\vert #1 \right\vert}

\usepackage{xspace}

\DeclareMathOperator{\distance}{\mathfrak{d}} % cu * nu pune powers and subscripts corect!!!

\theoremstyle{plain}
\newtheorem{theorem}{Theorem}
\newtheorem{lemma}{Lemma}

\theoremstyle{remark}
\newtheorem{remark}{Remark}
\newtheorem*{remark*}{Remark}

\newtheorem*{example*}{Example} %unnumbered

\usepackage{subfigure}  % package for subfigure formatting
\usepackage{float}
\usepackage[pdftex]{graphicx}
\graphicspath{{IMAGES/}{./}}
\usepackage[]{caption}
\usepackage{cite}

\makeatletter
\@namedef{subjclassname@2020}{\textup{2020} Mathematics Subject Classification}
\makeatother

\begin{document}

% % \usepackage{xspace}
% \newcommand{\propstar}{\textit{Property}~\eqref{PS}\xspace}
% \newcommand{\bp}{\mathbf{p}}
% % \newcommand{\fp}{\mathfrak{p}}

\title[Counterintuitive patterns on angles and distances]
{Counterintuitive patterns on angles and distances between lattice points in high
dimensional hypercubes}

\author[J. Anderson, C. Cobeli, A. Zaharescu]{Jack Anderson, Cristian Cobeli, 
Alexandru Zaharescu}

\address{
JA: Department of Mathematics,
University of Illinois at Urbana-Champaign,
Altgeld Hall, 1409 W. Green Street,
Urbana, IL, 61801, USA
}
\email{jacka4@illinois.edu}

\address{
CC: Simion Stoilow Institute of Mathematics of the Romanian Academy, 
P. O. Box 1-764, RO-014700 Bucharest, Romania}
\email{cristian.cobeli@gmail.com}

\address{
AZ: Department of Mathematics,
University of Illinois at Urbana-Champaign,
Altgeld Hall, 1409 W. Green Street,
Urbana, IL, 61801, USA and Simion Stoilow Institute of Mathematics of the Romanian Academy, 
P. O. Box 1-764, RO-014700 Bucharest, Romania}
\email{zaharesc@illinois.edu}  

% \date{July, 2, 2020}
\date{\today}
% 11A41 Primes
% 11B25 Arithmetic progressions (Sequences and sets)
% 11N13 Primes in progressions (Multiplicative Number Theory)
% `Dirichlet density' cat si `natural density' sau simplu in locul lor `density'

\begin{abstract}
Let $\cS$ be a finite set of integer points in $\RR^d$, which we assume has many symmetries, 
and let $P\in\RR^d$ be a fixed point. 
We calculate the distances from $P$ to the points in $\cS$ and compare the results.
In some of the most common cases, we find that they lead to unexpected conclusions
if the dimension is sufficiently large. 
For example, if $\cS$ is the set of vertices of a hypercube in $\RR^d$ 
and $P$ is any point inside, then almost all triangles $PAB$ with $A,B\in\cS$ are
almost equilateral. Or, if $P$ is close to the center of the cube, then almost all
triangles $PAB$ with $A\in \cS$ and $B$ anywhere in the hypercube are almost right triangles.
\end{abstract}
\maketitle

%%%%%%%%%%%%%%%%%%%%%%%%%%%%%%%%%%%%%%%%%%%%%%%%%%%%%%%%%%%%%%%
\section{Introduction}
\noindent
Recent developments in network communications~\cite{LQ2016, SH2010} or artificial intelligence~\cite{BS2021} 
have shed new light on studies of graphs and special models based on sets 
explored in combinatorial geometry
or related to lattice points in multidimensional 
spaces~\cite{AHK2001, ACZ2023, LM2023, Buc1986, ES1996, Hal1982, OO2015}.
Our object in this article is to present a few results related to the fact that
in high dimensional hypercubes, a random pick of lattice points to find some 
that are at an `exceptional distance' apart from each other has zero chance of 
success if the dimension goes to infinity. (Here, an \emph{exceptional distance} 
is any one that is different from the average.)

Let $\cS\subset\RR^d$, $d\ge 1$, be a finite set and let $\ba=(a_1,\dots,a_d)\in\RR^d$.
If we look from a distant point $\ba$ to the points in $\cS$, 
we find that they are all at about the same distance, which is 
the closer to a certain value the farther away from $\cS$ our point of view is.
On the contrary, if our point of view is close to $\cS$, even in $\cS$ or in its convex envelope,
we see the variety of distances ranging from zero to the diameter of $\cS$.
But what we observe is largely influenced by the size of the space and its dimensions.
Our goal here is to highlight some counterintuitive phenomena, some of them somehow 
related to the ones that have been studied for the set of lattice points visible from each other 
in a hypercube~\cite{ACZ2023}. 
In order to illustrate two of these phenomena, let us note that if we randomly pick triangles 
with vertices at the lattice points of a cube that is large enough, the likelihood of 
encountering a significant number of some special triangles is low. 
%%%%%%%%%%%%%%%%%%%%%%%%%%%%%%%%%%%%%%%%%%%%%%%%%%%%%%%%%%%%%%%%%%
% \setlength{\aboveaptionskip}{-99pt}
\setlength{\belowcaptionskip}{-3pt}
\begin{figure}[ht]
 \centering
\includegraphics[width=\textwidth]{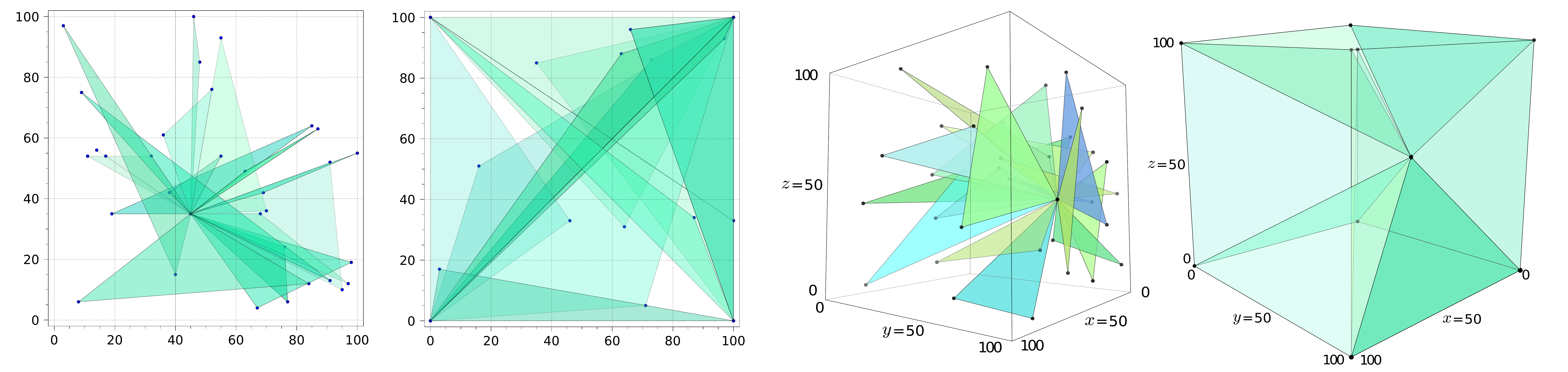}
 \vspace{-6mm}
\caption{Random triangles, in 2D and 3D,
with vertices of integer coordinates in $[0,100]$.
In each image, the triangles are chosen in such a way that they meet one of the following conditions:
A. All triangles have a common vertex, initially chosen randomly but later fixed.
B. All triangles have two vertices randomly chosen from the vertices of the cube, while the third vertex is free.
}
 \label{Fig2d3d}
 \end{figure}
 
For instance, in Figure~\ref{Fig2d3d}, we can see two type of selections, 
each with a distinct feature in dimensions $2$ and $3$.
The first type of choice conditions the triangles to have a common vertex, while the second one 
requires that two of the triangle's vertices be chosen randomly from the cube's vertices, 
while the third one remains unrestricted.
Then we can wonder, what are the odds of getting similar triangles in the first case or non-degenerate
isosceles triangles in the second case?
Certainly, the questions may appear uninteresting, as the answer is so small in both situations.

Furthermore, as the size of the cube and the dimension increases, 
the variety of these triangles increases immensely, and the attempt to randomly find the special 
ones seems completely in vain.
Despite this, the situation is not like that at all, but rather the complete opposite.
Thus, Theorem~\ref{ThAVVSimple} shows that, if the dimension of the hypercube becomes large enough,
then almost all triangles that have two vertices at the corners of the hypercube 
and the third a lattice point inside are almost isosceles.
And on the same note, if both the size of the hypercube and the dimension become sufficiently large, 
then Theorem~\ref{Theorem3} shows that almost all triangles with vertices anywhere 
on the lattice points of the hypercube, which have a certain common vertex, not only are nearly isosceles 
but also have a particular shape, being almost all almost similar.

To make things precise, let $N\ge 1$ be integer and let  $\cW=\cW(d,N)$ be the maximal hypercube of lattice points from $[0,N]^d$. 
Since we are interested both in the discrete case and in the limit process, 
a good coverage of the phenomenon is taken if we choose $\cS\subseteq \cW$. 
We measure the distance between points 
$\bv',\bv''\in\RR^d$
with the Euclidean distance
\begin{equation*}
    \distance(\bv',\bv'')=\big((v_1''-v_1')^2+\cdots+(v_d''-v_d')^2\big)^{1/2}
\end{equation*}
and, to compare with each other sizes from different dimensions, we use
the \emph{normalized distance}:
\begin{equation*}
    \distance_d(\bv',\bv'')=\frac{1}{\sqrt{d}N}\big((v_1''-v_1')^2+\cdots+(v_d''-v_d')^2\big)^{1/2}.
\end{equation*}
Then the normalized distance between two opposite vertices, the farthest
away points in~$\cW$, is $\distance_d\big((0,\dots,0),(N,\dots,N)\big)=1$.

In direct contrast, besides `\emph{the thickest}' hyperlane' $\cW$, we also consider 
'\emph{the thinnest}' one, that of dimension zero, the set of vertices of $[0,N]^d$, 
which we denote by $\cV=\cV(d,N)$.
For orientation in $\cW$ or around, a useful support from some point~$\ba$ turns out to be 
the distance from $\ba$ to the center of the cube,  $\bc$. That is why we denote
$r_{\ba}:=\distance_d(\ba,\bc)$.

From an arbitrary point $\ba$, in Sections~\ref{SectionAMV} and~\ref{SectionAMW}, we
find exact formulas for the average distances to points in $\cV$ or in $\cW$, respectively. 
Also, we calculate the second moments about these averages in both cases. They are the main tool
that allow us to draw catching properties that most pairs or triples of points in the hypercube have
in high dimensions.
We mention that similar procedures were used recently in other different settings.
For example, in a continuum case, in order to provide a framework 
for studying multifractal geometry, the authors of~\cite{AEHO2017} and~\cite{OR2019} 
study the average distance and the asymptotic behavior of higher moments of self-similar 
measures on self-similar subsets of $\RR$, and on graph-directed self-similar subsets of~$\RR$.
Corresponding characteristic properties of lattice points that are visible from each others were 
observed in~\cite{ACZ2023}.
Averages of relative distances from points in geometric figures were also the object of study in
the articles~\cite{MMP1999, LQ2016, Dun1997, BP2009, Bas2021}.

To exemplify our results, regarding, for example, to the vertices of the hypercube, 
one may ask what is the expected distance from them to a fixed arbitrary point $\ba$ 
and what is the probability that such a distance is close to the average.
In Section~\ref{SectionAVAV}, we show that, for any fixed point $\ba\in\cW$,
almost all vertices are at a normalized distance from $\ba$ that is close to
$\sqrt{1/4+r_{\ba}^2}$,
% $\sqrt{1/4+\distance_d^2(\ba,\bc)}$
% (where $\bc$ is the center of the hypercube), 
so long as the dimension $d$ is sufficiently large.
As a consequence, it follows that almost all triangles formed from~$\ba$ and two vertices of the hypercube
will be nearly an isosceles triangle, since the distances from~$\ba$ to each of the 
two vertices will both be close to the same value.
\medskip
%%%%%%%%%%%%%%%%%%%%%%%%%%%%%%%%%%%
\begin{theorem}\label{ThAVVSimple}
    For all $\varepsilon>0$, there exists an integer $d_\varepsilon$ such that, for all integers $d\geq d_\varepsilon$,
    $N\geq 1$, and any point $\ba\in\cW$, the proportion of triangles $(\ba,\bvi,\bvii)$
    such that
    \begin{equation*}
        \abs{\distance_d(\ba,\bvi)-\distance_d(\ba,\bvii)}\leq\varepsilon,
    \end{equation*}
    where $\bvi,\bvii\in\cV$, is greater than or equal to $1-\varepsilon$.
\end{theorem}

Another consequence arises from noticing that, for any vertex $\bv\in\cV$, the square of the normalized distance from the center
of the cube to $\bv$ is $1/4$. As a result, for almost all vertices $\bv$, the square of the distance from $\ba$ to $\bv$ is almost
the sum of the squares of the distances from $\bc$ to $\ba$ and from $\bc$ to $\bv$. Therefore, it is natural to ponder if
$(\ba,\bc,\bv)$ may be close to a right triangle, and in fact this is the case so long as $\ba$ is not too near to $\bc$. 

\begin{theorem}\label{ThACVSimple}
    For all $\varepsilon>0$, there exists an integer $d_\varepsilon$, and a function $f(d)\leq1/2$, such that for all
    integers $d\geq d_\varepsilon$, $N\geq1$, and any point $\ba\in\cW$ with $\distance_d(\ba,\bc)\geq f(d)$ (where
    $\bc$ is the center of the hypercube), the proportion of triangles $(\ba,\bc,\bv)$ 
    with $\bv\in\cV$ and whose angle~$\theta_{\bc}(\bv)$ at $\bc$ satisfies
    \begin{equation*}
        \abs{\cos\theta_{\bc}(\bv)} \leq \varepsilon,
    \end{equation*} 
    is greater than or equal to $1-\varepsilon$.
\end{theorem}
\noindent
Precise estimates and the effective bounds versions of Theorems~\ref{ThAVVSimple} 
and~\ref{ThACVSimple} are proved in Section~\ref{SectionVTriangles}.

In the second part of our manuscript, starting with Section~\ref{SectionAMW}, we turn our focus to looking at distances from a fixed point $\ba$ to any integer point $\bw$ in the cube.
We similarly find that almost all points $\bw\in\cW$ are at a normalized distance from $\ba$ 
which is close to $\sqrt{1/12+1/(6N)+\distance_d^2(\ba,\bc)}$, 
provided that the dimension~$d$ is sufficiently large. 
Furthermore, we will also show that almost all pairs of points in the cube are at a relative distance
close to $\sqrt{1/6+1/(3N)}$. As a consequence, we find that almost all triangles with 
one vertex at $\ba$ and the other two anywhere in $\cW$ are nearly identical. 
We summarise this fact in the following theorem, which, in explicit and effective 
form, we prove in Section~\ref{SectionSpacingsW}.
%%%%%%%%%%%%%%%%%%%%
\begin{theorem}\label{Theorem3}
For any $\varepsilon>0$, there exist positive integers $d_\varepsilon$, $N_\varepsilon$, 
such that, for all integers $d\geq d_\varepsilon$, $N\geq N_\varepsilon$,
    and any point $\ba\in\cW$, the proportion of triangles $(\ba,\bw_1,\bw_2)$, with
    $\bw_1,\bw_2\in \cW$, in which
    \begin{equation*}%\label{eqTrangleaw1w2}
        \abs{\distance_d(\bw_1,\bw_2)-\sdfrac{1}{\sqrt{6}}} \le \varepsilon, \text{ and }   
    \abs{\distance_d(\ba,\bw_j)-\sqrt{\sdfrac{1}{12}+r_{\ba}}} 
        \le \varepsilon, \text{ for $j=1,2$ }
    \end{equation*}
is greater than or equal to $1-\varepsilon$.
(Here, %In relation~\eqref{eqTrangleaw1w2},
$r_{\ba}=\distance_d(\ba,\bc)$ denotes the normalized distance from~$\ba$ 
to the center of $\cW$.)
\end{theorem}

For a probabilistic description of some natural expectations in high 
dimensional hypercubes we refer the reader to~\cite[Section 8]{ACZ2023}. 
It is a super-fast approach to the subject, although, there, the discussion is done in a 
continuum and the positions of both the observer and the viewed point are variable, while in this
paper, most of the time the observer has a fixed position.

%%%%%%%%%%%%%%%%%%%%%%%%%%%%%%%%%%%%%%%%%%%%%%%%%%%%%%%%%%%%%%%%%%%%%%%%%%%
\section{Distances between any fixed point and the vertices of the hypercube 
% -- The average and the second moment
}\label{SectionAMV}
\setcounter{subsection}{0} 
% BUG1 it doesn't reset the subsection counter automatically
% BUG2 the counter of subsection is set in afont that is larger than that of sections

% \noindent
For any $\ba=(a_1,\dots,a_d)\in\cW$, in the following we denote $\norm{\ba}^2:=a_1^2+\cdots+a_d^2$ and
$\normu{\ba}:=a_1+\cdots+a_d$.
Let $\cV$  denote the set of all vertices of $[0,N]^d$. 
This cube has $2^d$ vertices and each of them has components equal to $0$
or $N$. Notice that if~$\cV$ is seen as a subset of the set of lattice points $\cW$, then
no two vertices in $\cV$ are visible from each other, since there are always other points 
of integer coordinates in $[0,N]^d$ that interfere between them provided that $N\ge 2$.
The set of points in~$\cW$ that are visible from each other was the object of
study in~\cite{ACZ2023}.

%%%%%%%%%%%%%%%
\subsection{\texorpdfstring{The average $A_{\ba,\cV}(d,N)$}{The average A(a,V; d,N}}\label{SubsectionAV}
% \texorpdfstring{}
Let $\ba=(a_1,\dots,a_d)\in\RR^d$ be fixed and let $A_{\ba,\cV}(d,N)$ 
denote the average of the squares of the distances from~$\ba$ to all vertices $\bv\in\cV$. We have
\begin{equation*}
    \begin{split}
    A_{\ba,\cV}(d,N) &= \frac{1}{\#\cV}\sum_{\bv\in\cV}\distance^2(\bv,\ba)\\
    &=\frac{1}{2^d}\sum_{\bv\in\cV}
    \left((v_1-a_1)^2+\cdots+(v_d-a_d)^2\right)\\
    &=\frac{1}{2^d}\sum_{\bv\in\cV}\sum_{j=1}^{d}v_j^2
    -\frac{1}{2^{d-1}}\sum_{\bv\in\cV}\sum_{j=1}^{d}v_j a_j
    +\frac{1}{2^d}\sum_{\bv\in\cV}\sum_{j=1}^{d}a_j^2.
    \end{split}
\end{equation*}
For any fixed $j$, there are $2^{d-1}$ vertices $\bv\in\cV$ with
the $j$-th component equal to $0$, while the remaining ones have the
$j$-th component equal to $N$. Then
\begin{equation*}
    \begin{split}
    A_{\ba,\cV}(d,N) &=\frac{1}{2^d}\sum_{j=1}^{d}2^{d-1}N^2
    -\frac{1}{2^{d-1}}\sum_{j=1}^{d}a_j2^{d-1}N
    +\frac{1}{2^d}\norm{\ba}^2 2^d\\
    &=\frac{1}{2}dN^2-\normu{\ba}N+\norm{\ba}^2.
    \end{split}
\end{equation*}
We state the result in the next lemma.
%%%%%%%%%%%%%%%%%%%%%%%%%%%%%%%%%%%%
\begin{lemma}\label{LemmaAverageV}
Let $\cV$ be the set of vertices of the hypercube $[0,N]^d$, where 
$N\ge 1$ and $d\ge 1$ are integers. Let $\ba=(a_1,\dots,a_d)\in\RR^d$ be 
fixed.
Then, the average of all the squares of distances from $\ba$ to points in $\cV$ is
\begin{equation}\label{eqApV}
    A_{\ba,\cV}(d,N) =\frac{1}{2^d}\sum_{j=1}^{d}2^{d-1}N^2
    = \frac{1}{2}dN^2-\normu{\ba}N+\norm{\ba}^2.
\end{equation}
\end{lemma}
In particular, Lemma~\ref{LemmaAverageV} says that the average distance from
the origin to the vertices of $[0,N]^d$ equals $\sqrt{dN^2/2}$, 
which is the same as saying that the average normalized distance is $1/\sqrt{2}$.

% \medskip 
%%%%%%%%%%%%%%%%%%%%%%%%%%%%%%%%%
%\subsection{\texorpdfstring{The average $A_{\ba,\cV}(d,N)$}{The average A(a,V; d,N}}
\subsection{The second moment about the average distances to the vertices}\label{SubsectionM2V}
Starting with the definition of the second moment, which we denote by
$\fM_{2; \ba,\cV}(d,N)$, we rearrange the terms in its defining summation
to aggregate the average and make use of Lemma~\ref{LemmaAverageV}.
Thus, writing shortly $A_{\ba,\cV}$ instead of $A_{\ba,\cV}(d,N)$, we have:
\begin{equation}\label{eqM2aV1}
    \begin{split}
    \fM_{2; \ba,\cV}(d,N) &:= \frac{1}{\#\cV}\sum_{\bv\in\cV}
    \left(\distance^2(\bv,\ba)-A_{\ba,\cV}\right)^2\\
    &=\frac{1}{2^d}\sum_{\bv\in\cV}
    \left(\distance^4(\bv,\ba)
          -2\distance^2(\bv,\ba)A_{\ba,\cV}+A_{\ba,\cV}^2\right)\\
    &=\frac{1}{2^d}\left(\sum_{\bv\in\cV}\distance^4(\bv,\ba)
    -2A_{\ba,\cV}\sum_{\bv\in\cV}\distance^2(\bv,\ba)
    +\sum_{\bv\in\cV}A_{\ba,\cV}^2\right)\\
    &=\frac{1}{2^d}\cdot\Sigma_{\ba,\cV}-A_{\ba,\cV}^2.
    \end{split}
\end{equation}
To find the sum denoted by $\Sigma_{\ba,\cV}$ in~\eqref{eqM2aV1},  we write it explicitly:
\begin{equation}\label{eqM2aV11}
    \begin{split}
    \Sigma_{\ba,\cV} &:=\sum_{\bv\in\cV}\distance^4(\bv,\ba)
    % = \sum_{\bv\in\cV} \sum_{m=1}^d \sum_{n=1}^d(v_m-a_m)^2(v_n-a_n)^2
     = \sum_{\bv\in\cV} \sum_{m=1}^d \sum_{n=1}^d h(v_m,v_n, a_m,a_n),
    \end{split}
\end{equation}
where $h(v_m,v_n, a_m,a_n)=(v_m-a_m)^2(v_n-a_n)^2$ 
is the sum of the following nine monomials:
\begin{equation}\label{eqM2aV}
    \begin{split}
      h(v_m,v_n, a_m,a_n)=& v_m^2v_n^2-2v_m^2 v_n a_n + v_m^2a_n^2\\
      & -2v_m a_m v_n^2 + 4 v_m a_m v_n a_n -2v_m a_m a_n^2\\
       & +a_m^2 v_n^2 -2 a_m^2 v_n a_n + a_m^2a_n^2.
    \end{split}
\end{equation}
Next we take into account the contribution of each monomial in~\eqref{eqM2aV} 
to the corresponding sum in~\eqref{eqM2aV1}. For this we separate the 
group of the $d$ diagonal terms (those with $m=n$) from the group of the
$d^2-d$ off-diagonal terms, and then count the number of vertices 
with the non-zero components at the right place.
We have:
\begin{equation}\label{eqM2aV99a}
    \begin{split}
       S_1(\ba,\cV)&=\sum_{\bv\in\cV} \sum_{m=1}^d \sum_{n=1}^d v_m^2v_n^2
         = N^4 \left(d 2^{d-1}+ (d^2-d)2^{d-2}\right);\\
       S_2(\ba,\cV)&= \sum_{\bv\in\cV} \sum_{m=1}^d \sum_{n=1}^dv_m^2 v_n a_n
        =  N^3 \left(2^{d-1}\normu{\ba}+ (d-1)2^{d-2}\normu{\ba}\right);\\
       S_3(\ba,\cV)&=\sum_{\bv\in\cV} \sum_{m=1}^d \sum_{n=1}^d v_m^2a_n^2 
         = N^2 d2^{d-1}\norm{\ba}^2;\\
    \end{split}
\end{equation}
then
\begin{equation}\label{eqM2aV99b}
    \begin{split}
       S_4(\ba,\cV)&=S_2(\ba,\cV)= \sum_{\bv\in\cV} \sum_{m=1}^d \sum_{n=1}^dv_m a_m v_n^2 \\
        &=  N^3 \left(2^{d-1}\normu{\ba}+ (d-1)2^{d-2}\normu{\ba}\right);\\ 
    S_5(\ba,\cV)&= \sum_{\bv\in\cV} \sum_{m=1}^d \sum_{n=1}^dv_m a_m v_n a_n
        =  N^2 \left(2^{d-1}\norm{\ba}^2+ 2^{d-2}(\normu{\ba}^2-\norm{\ba}^2)\right);\\
        S_6(\ba,\cV)&= \sum_{\bv\in\cV} \sum_{m=1}^d \sum_{n=1}^dv_m a_m a_n^2
        =  N 2^{d-1}\normu{\ba}\norm{\ba}^2;\\
    \end{split}
\end{equation}
and
\begin{equation}\label{eqM2aV99c}
    \begin{split}
        S_7(\ba,\cV)&= S_3(\ba,\cV)=\sum_{\bv\in\cV} \sum_{m=1}^d \sum_{n=1}^da_m^2 v_n^2
        =  N^2 d 2^{d-1}\norm{\ba}^2;\\
         S_8(\ba,\cV)&= S_6(\ba,\cV)= \sum_{\bv\in\cV} \sum_{m=1}^d \sum_{n=1}^d a_m^2 v_n a_n
        =  N2^{d-1}\normu{\ba}\norm{\ba}^2;\\
         S_9(\ba,\cV)&= \sum_{\bv\in\cV} \sum_{m=1}^d \sum_{n=1}^d a_m^2 a_n^2
        =  2^d\norm{\ba}^4.\\
    \end{split}
\end{equation}
On adding the sums in~\eqref{eqM2aV99a},~\eqref{eqM2aV99b} and~\eqref{eqM2aV99c} as many times as indicated 
by the appearances of their defining monomials in~\eqref{eqM2aV}, we find that
the sum $\Sigma_{\ba,\cV}$ from~\eqref{eqM2aV11} is equal to
\begin{equation}\label{eqsigmaTot}
    \begin{split}
    \Sigma_{\ba,\cV} &= \big(S_1(\ba,\cV)-2S_2(\ba,\cV)+S_3(\ba,\cV)\big)\\
      &\phantom{ = (S_1(\ba,\cV)-} -\big(2S_4(\ba,\cV)-4S_5(\ba,\cV)+2S_6(\ba,\cV)\big)  \\
    &\phantom{ (S_1(\ba,\cV)-2S_2(\ba,\cV)+ S_3(  } +\big(S_7(\ba,\cV)-2S_8(\ba,\cV)+S_9(\ba,\cV)\big)\\
     &= 2^{d-2}\left(
     (d^2 + d)N^4 - 4(d+1)\normu{\ba}N^3 
     - 8\normu{\ba}\norm{\ba}^2 N \right.\\
     &\phantom{= 2^{d-2}( (d^2 + d)N^4 - }
     \left.+ 4\big(\normu{\ba}^2 + (d+1)\norm{\ba}^2\big)N^2 + 4\norm{\ba}^4\right).
    \end{split}
\end{equation}

Then, inserting the results from~\eqref{eqsigmaTot} and~\eqref{eqApV}
in formula~\eqref{eqM2aV1}, we arrive at a closed form expression 
for $\fM_{2; \ba,\cV}(d,N)$, which we state in the next lemma.
%%%%%%%%%%%%%%%%%%%%%%%%%%%%%%
\begin{lemma}\label{LemmaM2V}
Let $d,N \ge 1$ be integers and let $\cV$ be the set of vertices of the hypercube $[0,N]^d$. 
Then, the second moment about the mean $A_{\ba,\cV}(d,N)$ equals
\begin{equation*}%\label{eqM2VerticesFinal}
    \begin{split}
    \fM_{2; \ba,\cV}(d,N) =
     \frac{1}{\#\cV}\sum_{\bv\in\cV} \left(\distance^2(\bv,\ba)-A_{\ba,\cV}(d,N)\right)^2
    &= \frac{1}{4}dN^4-\normu{\ba} N^3+\norm{\ba}^2N^2.
    % 1/4*d*n^4 - a1*n^3 + a2*n^2
    \end{split}
\end{equation*}
\end{lemma}

%%%%%%%%%%%%%%%%%%%%%%%%%%%%%%%%%%%%%%%%
\section{The average of the squares and the square root of the average}\label{SectionAVAV}

Since the normalized second moment $\fM_{2;\ba,\cV}(d,N)/d^2N^4 = o(1)$ as $d\to\infty$, 
it follows 
% we obtain the following result, which shows 
that for any fixed $\ba\in\cW$, almost
all normalized distances from $\ba$ to the vertices of $\cW$ are close to $\sqrt{A_{\ba,\cV}(d,N)/dN^2}$. 
This is the object of the following theorem.

%%%%%%%%%%%%%%%%%%%%%%%%%%%%%%%%%%%%%%%%%%%
\begin{theorem}\label{TheoremV1}
Let $B_{\ba,\cV}:=A_{\ba,\cV}(d,N)/dN^2$ denote the average of the squares of
the normalized distances from $\ba$ to the vertices of $[0,N]^d$.
Let $\eta\in (0,1/2)$ be fixed.
Then, for any integers $d\ge 2$, $N\ge 1$, and any point $\ba\in\cW$, we have 
\begin{equation*}\label{eqPTV1}
       \sdfrac{1}{\#\cV} \#\left\{\bv\in \cV :
        \distance_d(\ba,\bv)\in \left[\sqrt{B_{\ba,\cV}}-\sdfrac{1}{d^\eta},\
                                    \sqrt{B_{\ba,\cV}}+\sdfrac{1}{d^\eta}    \right]
\right\}
\ge 1- \sdfrac{1}{d^{1-2\eta}}\,.
\end{equation*}
\end{theorem}

\begin{proof}
    Let $\eta, d, N, \ba$ be as in the statement of the theorem.
    Since
    \begin{equation*}
        -\normu{\ba}N+\norm{\ba}^2 \le 0,
    \end{equation*}
   from Lemma~\ref{LemmaM2V} we find  %\eqref{eqM2VerticesFinal} 
    that
    \begin{equation}\label{eqM2VerticesUpperBound}
    \begin{split}
        \sdfrac{\fM_{2;\ba,\cV}(d,N)}{d^2N^4} \le \sdfrac{1}{4d}\,.
    \end{split}
    \end{equation}
    On the other hand, for any parameters $b,T>0$,
    \begin{equation}\label{eqM2VerticesLowerBound}
    \begin{split}
        \sdfrac{\fM_{2;\ba,\cV}(d,N)}{d^2N^4} 
            &= \sdfrac{1}{\#\cV} \times\sum_{\bv\in\cV} 
            \left( \distance_d^2(\ba,\bv) - B_{\ba,\cV} \right)^2 \\
        &\ge \sdfrac{1}{\#\cV}\times \!\! \sum_{\substack{\bv\in\cV \\ \abs{\distance_d^2(\ba,\bv)-B_{\ba,\cV}} \ge \tfrac{1}{bT} }}
            \!\! \left(\distance_d^2(\ba,\bv)-B_{\ba,\cV}\right)^2 \\
        &\ge \sdfrac{1}{\#\cV} \times 
            \!\! \sum_{\substack{\bv\in\cV \\ \abs{\distance_d^2(\ba,\bv)-B_{\ba,\cV}} \ge \tfrac{1}{bT} }}
            \!\! \sdfrac{1}{b^2T^2} \,.
    \end{split}
    \end{equation}
    Then, on combining \eqref{eqM2VerticesUpperBound} and \eqref{eqM2VerticesLowerBound}, we see that
    \begin{equation}\label{eqThmV1SquaresBound}
    \begin{split}
        \sdfrac{1}{\#\cV} \#\left\{ \bv\in\cV \colon \abs{\distance_d^2(\ba,\bv)-B_{\ba,\cV}} \ge \sdfrac{1}{bT} \right\}
            \le \sdfrac{b^2T^2}{4d}.
    \end{split}
    \end{equation}
    Now, by Lemma~\ref{LemmaAverageV} and the definiton of $B_{\ba,\cV}$ in the hypothesis, we find that
    \begin{equation}\label{eqLowerBoundBV}
        \sqrt{B_{\ba,\cV}} = \sqrt{\sdfrac{1}{2} + \sdfrac{1}{dN^2}\left( \norm{\ba}^2 - N\normu{\ba} \right)}
            \geq \sqrt{\sdfrac{1}{2} - \sdfrac{1}{4}} = \sdfrac{1}{2}\,.
    \end{equation}
(Here we have taken into account the fact that the minimum of 
$\norm{\ba}^2-N\normu{\ba}$ is attained 
in the middle of the hypercube, which is a consequence of the fact that, 
independently on each of the $d$ coordinates, 
the minimum of $x\mapsto x^2-Nx$ is reached for $x=N/2$.)
Then, using inequality~\eqref{eqLowerBoundBV} 
and the fact that $\distance_d(\ba,\bv)\geq 0$, it follows that
    \begin{equation*}
    \begin{split}
        \abs{\distance_d^2(\ba,\bv) - B_{\ba,\cV}} &= \abs{\distance_d(\ba,\bv) - \sqrt{B_{\ba,\cV}}}\left( \distance_d(\ba,\bv) + \sqrt{B_{\ba,\cV}} \right)\\
            &\ge \sdfrac{1}{2}\abs{\distance_d(\ba,\bc) - \sqrt{B_{\ba,\cV}}}.
    \end{split}
    \end{equation*}
Therefore, we can tighten the restriction in the set on the left-hand side of
inequality~\eqref{eqThmV1SquaresBound} by taking $b=2$, and, as a consequence,
we find that
    \begin{equation*}%\label{eqThmV1Proved}
    \begin{split}
         \sdfrac{1}{\#\cV} \#\left\{ \bv\in\cV \colon \abs{\distance_d(\ba,\bv)-\sqrt{B_{\ba,\cV}}} \ge \sdfrac{1}{T} \right\}
            \le \sdfrac{T^2}{d}.
    \end{split}
    \end{equation*}
Finally, we take $T=d^{\eta}$ and then see that this completes the proof of Theorem~\ref{TheoremV1}.
\end{proof}

%%%%%%%%%%%%%%%%%%%%%%%%%%%%%%%%%%%%%%%%%%%%%%%%%%%%%%%%%%%%

\section{Triangles involving the Vertices of the Hypercube}\label{SectionVTriangles}
\setcounter{subsection}{0} 

In this section we analyze the set of triangles in which at least one vertex is a corner of the
hypercube.  
We count them to see how many of them are close or away from the average.

%%%%%%%%%%%%%
\subsection{Triangles formed from any fixed point and two vertices of the cube}
\noindent
From Theorem~\ref{TheoremV1}, it will follow that almost all pairs of distinct vertices $(\bvi,\bvii)\in\cV^2$ are each at a distance close to
$\sqrt{B_{\ba,\cV}}$ from $\ba$. As a result, one can say that almost all triangles formed by $\ba$ and two vertices are `almost isosceles'. 
If we denote by $\cT_{\ba,\cV^2}\subset\cV^2$ the set of all pairs of vertices $(\bvi,\bvii)$, which form together with $\ba$ a non-degenerate 
 triangle (that is, triangles with distinct vertices), then
\begin{equation}\label{eqv1v2}
\begin{split}
    &\sdfrac{1}{\#\cT_{\ba,\cV^2}}
    \#\left\{ (\bvi,\bvii)\in\cT_{\ba,\cV^2} \colon \distance_d(\ba,\bvi), \distance_d(\ba,\bvii) \in
        \left[ \sqrt{B_{\ba,\cV}}-\sdfrac{1}{d^\eta}, \sqrt{B_{\ba,\cV}}+\sdfrac{1}{d^\eta}\right]\right\} \\
    &\phantom{(\bvi, }
    \ge \sdfrac{1}{\#\cV}\left( \#\left\{ \bvi\in\cV \colon \distance_d(\ba,\bvi) \in 
        \left[ \sqrt{B_{\ba,\cV}}-\sdfrac{1}{d^\eta}, \sqrt{B_{\ba,\cV}}+\sdfrac{1}{d^\eta}\right] \right\} 
        - 1 \right)\\
    &\phantom{(\bvi, \ge}\times\sdfrac{1}{\#\cV}\left( \#\left\{ \bvii\in\cV \colon \distance_d(\ba,\bvii) \in 
        \left[ \sqrt{B_{\ba,\cV}}-\sdfrac{1}{d^\eta}, \sqrt{B_{\ba,\cV}}+\sdfrac{1}{d^\eta}\right] \right\}
        - 2 \right),
\end{split}
\end{equation}
where we subtract $1$ and $2$, respectively, from the two terms in the right-hand side of
\eqref{eqv1v2} to account for the possibilities that $\ba, \bvi, \bvii$ 
form a degenerate triangle.

From Theorem~\ref{TheoremV1}, we see that the right-hand side of \eqref{eqv1v2} is bounded below by
\begin{equation*}
\begin{split}
    \left(1-\sdfrac{1}{d^{1-2\eta}}-\sdfrac{1}{2^d}\right)
    &
    \left(1-\sdfrac{1}{d^{1-2\eta}}-\sdfrac{2}{2^d}\right)\\
     &=1-\sdfrac{2}{d^{1-2\eta}} + \sdfrac{1}{d^{2-4\eta}}-\sdfrac{3}{2^d}+\sdfrac{3}{2^d d^{1-2\eta}}+\sdfrac{2}{2^{2d}} \\
    &\ge 1-\sdfrac{2}{d^{1-2\eta}},
\end{split}
\end{equation*}
for $d\geq 8$, since in that range
$1/d^{2-4\eta}-3/2^d \ge 0$. We now arrive at the following theorem on isosceles triangles.
%%%%%%%%%%%%%%%%%%%%%%%%%%%%%%
\begin{theorem}\label{ThVIsosceles}
Let  $\ba\in\cW$ be fixed and let $\cT_{\ba,\cV^2}$ 
denote the set of triangles with distinct vertices $\ba$ and $\bv_1,\bv_2\in\cV$.
Let $\eta\in(0,1/2)$ be fixed. Then, for any integers $d\ge8$ and $N\ge1$, we have
    \begin{equation*}
        \frac{1}{\#\cT_{\ba,\cV^2}}
        \#\left\{ (\bvi,\bvii)\in\cT_{\ba,\cV^2} \colon 
            \abs{\distance_d(\ba,\bvi)-\distance_d(\ba,\bvii)} \le \sdfrac{2}{d^\eta} \right\}
        \ge 1-\sdfrac{2}{d^{1-2\eta}}.
    \end{equation*}
\end{theorem}

Seeing that almost every such triangle is almost isosceles, we may wonder if any of 
these triangles can be equilateral, or perhaps right triangle. 

Let $\bc=\big(\frac{N}{2},\dots,\frac{N}{2}\big)$ be the center of the hypercube.
Notice that $\bc$ may belong or not to~$\cW$, but in any case, if $N$ had been odd,
than the distance from $\bc$ to a point in $\cW$ would have been not greater than
$\sqrt{d}/2$. This is the same as saying that the normalized distance from
$\bc$ to $\cW$ is at most $1/(2N)$ and we may make reasoning with a point
with integer coordinates that is close to $\bc$ instead of $\bc$.
For simplicity we may assume here that $N$ is even, but this is not necessary, 
since in fact, in the proofs of Theorems~\ref{TheoremV1} and~\ref{ThVIsosceles},
we did not make use of the fact that the coordinates of $\ba$ are integers.

Note that all vertices from $\cV$ are equally far from the center and
\begin{equation}\label{eqDistanceCV}
    \distance_d(\bv,\bc)=\sdfrac{1}{\sqrt{d}N}
    \Big(\sum_{1\le j\le d}\big(N/2\big)^2\Big)^{1/2}
    = \sdfrac{1}{2}, \text{ for $\bv\in\cV$,}
\end{equation}
while for arbitrary $\ba=(a_1,\dots,a_d)$, the normalized distance to $\bc$ is
\begin{equation}\label{eqDistanceCba}
    \distance_d(\ba,\bc)=\sdfrac{1}{\sqrt{d}N}
    \Big(\sum_{1\le j\le d}\big(a_j-N/2\big)^2\Big)^{1/2}
    = \sdfrac{1}{\sqrt{d}N}\left(\sdfrac{dN^2}{4}-\normu{\ba}N+\norm{\ba}^2\right)^{1/2}.
\end{equation}

Now let us point out the following two observations.
%%%%%%%%%%%%%%%%%%%%%%%%%%%%%%%
\begin{remark}\label{Remark1}
\texttt{(1)} Taking $\ba$ in $\cV$, Theorem~\ref{TheoremV1} tells us that the normalized
distance between almost any two vertices in $\cV$ is close to $1/\sqrt{2}$.

\texttt{(2)} 
By Lemma~\ref{LemmaAverageV}, the normalized average of the squares of distances from
$\ba$ to vertices in $\cV$ is 
$ B_{\ba,\cV}=A_{\ba,\cV}/(dN^2) =1/2-\normu{\ba}/(dN)+\norm{\ba}^2/(dN^2)$.
Then, by~\eqref{eqDistanceCV} and~\eqref{eqDistanceCba} this can further be expressed as
\begin{equation}\label{eqPythagoras}
    B_{\ba,\cV} = \sdfrac{1}{4} + \left(\sdfrac{1}{4}-\sdfrac{\normu{\ba}}{dN}+\sdfrac{\norm{\ba}^2}{dN^2}\right) 
            = \distance_d^2(\bv,\bc)+\distance_d^2(\ba,\bc), 
            \text{ for any $\bv\in\cV$}.
\end{equation}
In particular,~\eqref{eqPythagoras} shows that the average $B_{\ba,\cV}$
depends only on the normalized distance $\distance_d(\ba,\bc)$, which
we shall further denote by $r_{\ba}$.
\end{remark}
\smallskip

On combining Theorem~\ref{ThVIsosceles},~\eqref{eqDistanceCV}, 
and the observations from Remark~\ref{Remark1},
we see that almost all triangles in  $\cT_{\ba,\cV^2}$ have one side lengths close to
$1/\sqrt{2}$, while the other are both close to $\sqrt{1/4+r_{\ba}^2}$.
Since, by normalization, we know that $0\le r_{\ba} \le 1/2$, it follows that
\begin{equation}\label{eqBoundra}
    \sdfrac{1}{2} \le \sqrt{\sdfrac{1}{4}+r_{\ba}^2} \le \sdfrac{1}{\sqrt{2}}.
\end{equation}
In particular, if $r_{\ba}=1/2$, which occurs when $\ba$ is a vertex, 
we see that almost all triangles have each of their side lengths close to $1/\sqrt{2}$.
In other words, almost all triangles formed by three vertices of the hypercube 
are `almost equilateral'.
On the other hand, if $r_{\ba}=0$, which occurs when $\ba=\bc$ is at the center of the
hypercube, we see that almost all triangles have side lengths close to $1/\sqrt{2}$, $1/2,$ and $1/2$, respectively, that is, they are almost isosceles with an almost 
right angle in $\bc$.
Making this more explicit, we argue similarly as in \eqref{eqv1v2} to find the proportion of non-degenerate triangles $(\ba,\bvi,\bvii)$ such that
\begin{equation}\label{eqCommonTriangleConditions}
\begin{split}
    \distance_d(\ba,\bvi), \distance_d(\ba,\bvii) &\in \left[ \sqrt{B_{\ba,\cV}}-\sdfrac{1}{d^\eta}, \sqrt{B_{\ba,\cV}}+\sdfrac{1}{d^\eta} \right],
    \text{ and } \\
    \distance_d(\bvi,\bvii) &\in \left[ \sdfrac{1}{\sqrt{2}}-\sdfrac{1}{d^\eta}, 
    \sdfrac{1}{\sqrt{2}}+\sdfrac{1}{d^\eta} \right].
\end{split}
\end{equation}
Firstly, for any $0<\eta<1/2$, from Theorem~\ref{TheoremV1}, we know that 
for any vertex $\bv\in\cV$, the proportion of vertices $\bvi\in\cV$ such that
\begin{equation*}
\begin{split}
    \distance_d(\ba,\bvi) \not\in \left[ \sqrt{B_{\ba,\cV}}-\sdfrac{1}{d^\eta}, \sqrt{B_{\ba,\cV}}+\sdfrac{1}{d^\eta} \right] \text{ and }
    \distance_d(\bvi,\bv) \not\in \left[ \sdfrac{1}{\sqrt{2}}-\sdfrac{1}{d^\eta}, \sdfrac{1}{\sqrt{2}}+\sdfrac{1}{d^\eta} \right],
\end{split}
\end{equation*}
is bounded above by
\begin{equation*}
    \sdfrac{1}{d^{1-2\eta}}+\sdfrac{1}{d^{1-2\eta}} = \sdfrac{2}{d^{1-2\eta}}.
\end{equation*}
Therefore, where $\bv\in\cV$ can be taken to be any vertex, the proportion of non-degenerate 
triangles formed by distinct vertices $(\bvi,\bvii,\ba)$,
which satisfy conditions in~\eqref{eqCommonTriangleConditions}, is bounded below by
\begin{equation*}
\begin{split}
    &\sdfrac{1}{\#\cV}\bigg(\#\left\{ \bvi\in\cV \colon 
        \distance_d(\ba,\bvi) \in \left[ \sqrt{B_{\ba,\cV}}-\sdfrac{1}{d^\eta}, \sqrt{B_{\ba,\cV}}+\sdfrac{1}{d^\eta} \right] 
        \right\} - 1 \bigg) \\
    \times&\sdfrac{1}{\#\cV}\bigg(\#\bigg\{ \bvii\in\cV \colon
        \distance_d(\ba,\bvii) \in \left[ \sqrt{B_{\ba,\cV}}-\sdfrac{1}{d^\eta}, \sqrt{B_{\ba,\cV}}+\sdfrac{1}{d^\eta} \right], \text{ and }\\
    &\phantom{\qquad\qquad\qquad\qquad\qquad\quad} 
        \distance_d(\bv,\bvii) \in \left[\sdfrac{1}{\sqrt{2}}-\sdfrac{1}{d^\eta}, \sdfrac{1}{\sqrt{2}}+\sdfrac{1}{d^\eta} \right]
        \biggr\} - 2 \bigg) \\
    &\ge \left( 1-\sdfrac{1}{d^{1-2\eta}}-\sdfrac{1}{2^d}\right)
        \left( 1-\sdfrac{2}{d^{1-2\eta}}-\sdfrac{2}{2^d}\right) \\
    &\ge 1-\sdfrac{3}{d^{1-2\eta}},
\end{split}
\end{equation*}
for $d\ge 6$.
As a consequence, we now have the following theorem.
\begin{theorem}\label{ThVSimilarTriangles}
Let $\cT_{\ba,\cV^2}$ be the set of triangles with distinct vertices $\ba$ and $\bv_1,\bv_2\in\cV$.
    Let $\eta\in(0,1/2)$ be fixed. Then, for any integers $d\ge6$, $N\ge1$, and any point
    $\ba\in\cW$, we have
    \begin{equation*}
        \sdfrac{1}{\#\cT_{\ba,\cV^2}} 
        \#\left\{ (\bvi,\bvii)\in\cT_{\ba,\cV^2} 
        :
            \ba, \bvi, \bvii \textup{ satisfy \eqref{eqCommonTriangleConditions}} \right\}
        \ge 1-\sdfrac{3}{d^{1-2\eta}}.
    \end{equation*}
\end{theorem}

%%%%%%%%%%%%

\subsection{Triangles formed by the center, any fixed point, and a vertex of the cube}
Counting triangles with one vertex in $\bc$ and another in $\cV$, and then sorting by the angle at 
the center, can be done by a somewhat winded combinatorial computation, but here
we can take a shortcut using the effective result from Section~\ref{SectionAVAV}.

By Theorem~\ref{TheoremV1} and~\eqref{eqPythagoras}, we know that if $\ba\in\cW$ is 
fixed, then for any $0<\eta<1/2$, for almost all vertices~$\bv\in\cV$, we have
\begin{equation*}
    \distance_d(\ba,\bv) = \sqrt{\distance_d^2(\bc,\bv)^2 + \distance_d^2(\bc,\ba)} + E(d^{-\eta}),
\end{equation*}
where $\abs{E(x)}\le x$. This suggests that the triangle with vertices
$\bc,\ba,\bv$
may be close to a right triangle.
Letting $\theta$ denote the angle at $\bc$, using the bound~\eqref{eqBoundra}, we see that
\begin{equation*}
\begin{split}
    \abs{\cos(\theta)} &= 
    \abs{\sdfrac{1}{4}+r_{\ba}^2-\left( \sqrt{\sdfrac{1}{4}+r_{\ba}^2} + E(d^{-\eta}) \right)^2}
        \Big(2\cdot\sdfrac{1}{2}\cdot r_{\ba}\Big)^{-1}    \\
    &\le \left(\sdfrac{2}{\sqrt{2}}\abs{E(d^{-\eta})} +\abs{E(d^{-\eta})}^2\right)r_{\ba}^{-1}
    \\
    &\le (\sqrt{2}+1)d^{-\eta}r_{\ba}^{-1}.
\end{split}
\end{equation*}
Picking any $\ba\in\cW$ such that $r_{\ba}\ge d^{\gamma-\eta}$ for some $\gamma>0$, we find that
\begin{equation*}
    \abs{\cos(\theta)} \le \big(\sqrt{2}+1\big)d^{-\gamma}.
\end{equation*}
Bearing in mind that we have an upper bound $r_{\ba}\le 1/2$, one should avoid picking~$\gamma$ 
too large or $d$ too small with respect to $\eta$, otherwise no such $\ba$
would exist. As a consequence we have obtained the following result.
%%%%%%%%%%%%%%%%%%%%%%%%%%%%%%%%%%%%%%%%%%
\begin{theorem}\label{ThTriangleacv}
Let $\eta\in(0,1/2)$ be fixed. 
Suppose $d>2^{1/\eta}$ is an integer and let 
\mbox{$\gamma\in\big(0,\,\eta-\frac{\log 2}{\log d}\big]$.}
% \mbox{$\gamma\in(0,\eta-\log 2/\log d]$.}
Further, for any integer $N\geq1$, with $\bc=\big(\frac{N}{2},\dots,\frac{N}{2}\big)$ 
denoting the center of the hypercube, fix any point $\ba\in\cW$ such that $\distance_d(\ba,\bc)\geq d^{\gamma-\eta}$.
    Then,
    \begin{equation*}
        \frac{1}{\#\cV}\#\left\{ \bv\in\cV : \abs{\cos(\theta_{\bv})} \leq \big(\sqrt{2}+1\big)d^{-\gamma} \right\}
            \ge 1-\sdfrac{1}{d^{1-2\eta}},
    \end{equation*}
    where, for any vertex $\bv\in\cV$, $\theta_{\bv}$ is the angle between the lines going from $\bc$
    to $\ba$ and $\bv$ respectively.
\end{theorem}
In plain words, Theorem~\ref{ThTriangleacv} says that as long as $\ba$ is not too close to 
the center of the cube, almost all triangles formed by $\ba$, $\bc$, and a vertex of the cube
are almost right triangles.

%%%%%%%%%%%%%%%%%%%%%%%%%%%%%%%%%%%%%%
\section{The spacings between a fixed point and the lattice points in the hypercube}\label{SectionAMW}

In this section we first calculate the mean distance from a fixed points to all the lattice
points in $\cW$. Afterwards, we use the result to find the second moment about the mean
of these distances. 
This is the farthest opposite case in terms of the dimension from the problem dealt with before.
Here, the whole hypercube of lattice points plays the previous role of the vertices.

%%%%%%%%%%%%%%%%%%%%%%%%%%%%%%%%%%%%%%%%%%%%%%%%%%%%%%%%%%%%%%%%%
\subsection{\texorpdfstring{The average $A_{\ba,\cW}(d,N)$}{The average A(a,W; d,N}}

Let $\ba=(a_1,\dots,a_d)\in\RR^d$ be fixed and denote
\begin{equation*}
   \begin{split}
   A_{\ba,\cW}(d,N):=\frac{1}{\#\cW}\sum_{\bv\in\cW}  \distance^2 (\ba,\bv)\,.
   \end{split}   
\end{equation*}
Using the definitions and rearranging the terms, we find that
\begin{equation}\label{eqAW1}
   \begin{split}
   A_{\ba,\cW}(d,N)&=\frac{1}{\#\cW}\sum_{\bv\in\cW} \sum_{j=1}^d (v_j-a_j)^2\\
              &=\frac{1}{\#\cW}\sum_{\bv\in\cW}\sum_{j=1}^dv_j^2
              -\frac{2}{\#\cW}  \sum_{\bv\in\cW} \sum_{j=1}^d a_j v_j +\norm{\ba}^2.
   \end{split}   
\end{equation}
Here, changing the order of summation, the sum of the squares is
\begin{equation}\label{eqAW1a}
    \begin{split}
    \sum_{j=1}^d\sum_{\bv\in\cW}v_j^2 = d\sum_{\bv\in\cW}v_1^2
    =d(N+1)^{d-1}\sum_{v=0}^{N}v^2
%     =d(N+1)^{d-1}\frac{N(N+1)(2N+1)}{6}.
    =\frac{dN(N+1)^d(2N+1)}{6}.
    \end{split}    
\end{equation}
In the same way, the mixed sum in~\eqref{eqAW1} can be written as
\begin{equation}\label{eqAW1b}
    \begin{split}
    \sum_{j=1}^d\sum_{\bv\in\cW}a_jv_j 
    =\normu{\ba}(N+1)^{d-1}\sum_{v=0}^{N}v
%     =\normu{\ba}(N+1)^{d-1}\frac{N(N+1)}{2}.
    =\normu{\ba}\frac{N(N+1)^d}{2}.
    \end{split}    
\end{equation}
On inserting the results~\eqref{eqAW1a} and~\eqref{eqAW1b} in~\eqref{eqAW1} we find 
a closed form expression for $A_{\ba,\cW}(d,N)$, which we state in the next lemma.

%%%%%%%%%%%%%%%%%%%%%%
\begin{lemma}\label{LemmaAW}
Let $d,N \ge 1$ be integers, and let $\ba\in\RR^d$ be fixed. 
Let $\cW$ be the set of lattice points in $[0,N]^d$.
Then, the average of all squares of distances from $\ba$ to points in $\cW$  is
\begin{equation}\label{eqAW2}
   \begin{split}
   A_{\ba,\cW}(d,N)& =\frac{1}{\#\cW}\sum_{\bv\in\cW}  \distance^2 (\ba,\bv)
   =\frac{dN(2N+1)}{6}-\normu{\ba}{N}+\norm{\ba}^2.
   \end{split}   
\end{equation}

\end{lemma}

%%%%%%%%%%%%%%%%%%%%%%%%%%%%%%%%%%%%%%%%%%%%%%%%%%
\subsection{The second moment about the mean}

The second moment about the mean for the whole hypercube, denoted by $A_{\ba,\cW}=A_{\ba,\cW}(d,N)$, 
is defined by
\begin{equation*}%\label{eqM2aW1}
    \begin{split}
    \fM_{2; \ba,\cW}(d,N) & := \frac{1}{\#\cW}\sum_{\bv\in\cW}
    \left(\distance^2(\bv,\ba)-A_{\ba,\cW}\right)^2.
    \end{split}
\end{equation*}
Rearranging the terms on the summation, we may rewrite $\fM_{2; \ba,\cW}$ as
\begin{equation}\label{eqM2aW1}
    \begin{split}
    \fM_{2; \ba,\cW}(d,N) &=\ltfrac{1}{(N+1)^d}\sum_{\bv\in\cW}
    \left(\distance^4(\bv,\ba)
          -2\distance^2(\bv,\ba)A_{\ba,\cW}+A_{\ba,\cW}^2\right)\\
    &=\ltfrac{1}{(N+1)^d}\bigg(
         \sum_{\bv\in\cW}\distance^4(\bv,\ba)
    -2A_{\ba,\cW}\sum_{\bv\in\cW}\distance^2(\bv,\ba)
    +\sum_{\bv\in\cW}A_{\ba,\cW}^2\bigg)\\
    &=\ltfrac{1}{(N+1)^d}\cdot\Sigma_{\ba,\cW}-A_{\ba,\cW}^2.
    \end{split}
\end{equation}
Here the terms collected in $\Sigma_{\ba,\cW}$ are
the analogs of that from relation~\eqref{eqM2aV11}, so that their sum is 
\begin{equation}\label{eqM2aW17}
    \begin{split}
    \Sigma_{\ba,\cW} &=\sum_{\bv\in\cW}\distance^4(\bv,\ba)
     = \sum_{\bv\in\cW} \sum_{m=1}^d \sum_{n=1}^d h(v_m,v_n, a_m,a_n),
    \end{split}
\end{equation}
where $h(v_m,v_n, a_m,a_n)=(v_m-a_m)^2(v_n-a_n)^2$ 
% % $h(v_m,v_n, a_m,a_n)=(v_m-a_m)^2(v_n-a_n)^2$ 
is the same sum of nine monomials from~\eqref{eqM2aV}.

Next we calculate the contribution of each of the nine type of terms to the total sum.
In the process, we change the order of summation and take care if the terms are on the 
diagonal (that is, if $m=n$) or not. 
We denote by $T_k(N)$ the sum of the first $N$ $k$-powers of positive integers, 
that is, $T_k(N)=1^k+2^k+\cdots +N^k$.
Thus, we obtain:
\begin{equation}\label{eqM2aW99a}
    \begin{split}
       S_1(\ba,\cW)&=\sum_{\bv\in\cW} \sum_{m=1}^d \sum_{n=1}^d v_m^2v_n^2\\
         &= d(N+1)^{d-1}T_4(N) +(d^2-d)(N+1)^{d-2}T_2^2(N);\\
       S_2(\ba,\cW)&= \sum_{\bv\in\cW} \sum_{m=1}^d \sum_{n=1}^dv_m^2 v_n a_n\\
        &= \normu{\ba}(N+1)^{d-1}T_3(N) + \normu{\ba}(d-1)(N+1)^{d-2}T_1(N)T_2(N);\\
       S_3(\ba,\cW)&=\sum_{\bv\in\cW} \sum_{m=1}^d \sum_{n=1}^d v_m^2a_n^2 
         = \norm{\ba}^2 d (N+1)^{d-1}T_2(N) ;\\
    \end{split}
\end{equation}
then
\begin{equation}\label{eqM2aW99b}
    \begin{split}
    %   S_1(\ba,\cW)&=\sum_{\bv\in\cW} \sum_{m=1}^d \sum_{n=1}^d v_m^2v_n^2\\
    %      &= d(N+1)^{d-1}S_4(N) +(d^2-d)(N+1)^{d-2}\big(S_2(N)\big)^2;\\
    %   S_2(\ba,\cW)&= \sum_{\bv\in\cW} \sum_{m=1}^d \sum_{n=1}^dv_m^2 v_n a_n\\
    %     &= \normu{\ba}(N+1)^{d-1}S_3(N) + \normu{\ba}(d-1)(N+1)^{d-2}S_1(N)S_2(N);\\
    %   S_3(\ba,\cW)&=\sum_{\bv\in\cW} \sum_{m=1}^d \sum_{n=1}^d v_m^2a_n^2 
    %      = \norm{\ba}^2 d (N+1)^{d-1}S_2(N) ;\\
       S_4(\ba,\cW)&=S_2(\ba,\cW)= \sum_{\bv\in\cW} \sum_{m=1}^d \sum_{n=1}^dv_m a_m v_n^2 \\
        &= \normu{\ba}(N+1)^{d-1}T_3(N) + \normu{\ba}(d-1)(N+1)^{d-2}T_1(N)T_2(N);\\ 
    S_5(\ba,\cW)&= \sum_{\bv\in\cW} \sum_{m=1}^d \sum_{n=1}^dv_m a_m v_n a_n\\
        &=  \norm{\ba}^2(N+1)^{d-1}T_2(N)+\left(\normu{\ba}^2-\norm{\ba}^2\right)(N+1)^{d-2}T_1^2(N);\\
        S_6(\ba,\cW)&= \sum_{\bv\in\cW} \sum_{m=1}^d \sum_{n=1}^dv_m a_m a_n^2
        =  \normu{\ba}\norm{\ba}^2 (N+1)^{d-1}T_1(N);\\
        % S_7(\ba,\cW)&= S_3(\ba,\cW)=\sum_{\bv\in\cW} \sum_{m=1}^d \sum_{n=1}^da_m^2 v_n^2
        % =  \norm{\ba}^2d(N+1)^{d-1}S_2(N);\\
        %  S_8(\ba,\cW)&= S_6(\ba,\cW)= \sum_{\bv\in\cW} \sum_{m=1}^d \sum_{n=1}^d a_m^2 v_n a_n
        % =  \normu{\ba}\norm{\ba}^2(N+1)^{d-1}S_1(N);\\
        %  S_9(\ba,\cW)&= \sum_{\bv\in\cW} \sum_{m=1}^d \sum_{n=1}^d a_m^2 a_n^2
        % =  \norm{\ba}^4(N+1)^d.\\
    \end{split}
\end{equation}
and
\begin{equation}\label{eqM2aW99c}
    \begin{split}
    %   S_1(\ba,\cW)&=\sum_{\bv\in\cW} \sum_{m=1}^d \sum_{n=1}^d v_m^2v_n^2\\
    %      &= d(N+1)^{d-1}S_4(N) +(d^2-d)(N+1)^{d-2}\big(S_2(N)\big)^2;\\
    %   S_2(\ba,\cW)&= \sum_{\bv\in\cW} \sum_{m=1}^d \sum_{n=1}^dv_m^2 v_n a_n\\
    %     &= \normu{\ba}(N+1)^{d-1}S_3(N) + \normu{\ba}(d-1)(N+1)^{d-2}S_1(N)S_2(N);\\
    %   S_3(\ba,\cW)&=\sum_{\bv\in\cW} \sum_{m=1}^d \sum_{n=1}^d v_m^2a_n^2 
    %      = \norm{\ba}^2 d (N+1)^{d-1}S_2(N) ;\\
    %   S_4(\ba,\cW)&=S_2(\ba,\cW)= \sum_{\bv\in\cW} \sum_{m=1}^d \sum_{n=1}^dv_m a_m v_n^2 \\
    %     &= \normu{\ba}(N+1)^{d-1}S_3(N) + \normu{\ba}(d-1)(N+1)^{d-2}S_1(N)S_2(N);\\ 
    % S_5(\ba,\cW)&= \sum_{\bv\in\cW} \sum_{m=1}^d \sum_{n=1}^dv_m a_m v_n a_n\\
    %     &=  \norm{\ba}^2(N+1)^{d-1}S_2(N)+\left[\normu{\ba}^2-\norm{\ba}^2\right](N+1)^{d-2}\big(S_1(N)\big)^2;\\
    %     S_6(\ba,\cW)&= \sum_{\bv\in\cW} \sum_{m=1}^d \sum_{n=1}^dv_m a_m a_n^2
    %     =  \normu{\ba}\norm{\ba}^2 (N+1)^{d-1}S_1(N);\\
        S_7(\ba,\cW)&= S_3(\ba,\cW)=\sum_{\bv\in\cW} \sum_{m=1}^d \sum_{n=1}^da_m^2 v_n^2
        =  \norm{\ba}^2d(N+1)^{d-1}T_2(N);\\
         S_8(\ba,\cW)&= S_6(\ba,\cW)= \sum_{\bv\in\cW} \sum_{m=1}^d \sum_{n=1}^d a_m^2 v_n a_n
        =  \normu{\ba}\norm{\ba}^2(N+1)^{d-1}T_1(N);\\
         S_9(\ba,\cW)&= \sum_{\bv\in\cW} \sum_{m=1}^d \sum_{n=1}^d a_m^2 a_n^2
        =  \norm{\ba}^4(N+1)^d.\\
    \end{split}
\end{equation}
On combining~\eqref{eqM2aW99a},~\eqref{eqM2aW99b},~\eqref{eqM2aW99c}, and
\eqref{eqM2aV}, we find that $\Sigma_{\ba,\cW}$ from~\eqref{eqM2aW17} equals
\begin{equation}\label{eqsigmaaWTot}
    \begin{split}
    \Sigma_{\ba,\cW} =& 
    \big(S_1(\ba,\cW)-2S_2(\ba,\cW)+S_3(\ba,\cW)\big)\\
    &\phantom{= \big(S_1(\ba,\cW)\;}
    -\big( 2S_4(\ba,\cW) -4S_5(\ba,\cW)+2S_6(\ba,\cW)\big) \\
    &\phantom{=}+ \big(S_7(\ba,\cW)-2S_8(\ba,\cW)+S_9(\ba,\cW)\big)\\
     =& (N+1)^d\biggl( \left(\sdfrac{1}{9}d^2-\sdfrac{4}{45}d \right)N^4
        +\left( \sdfrac{1}{9}d^2 + \sdfrac{17}{90}d + \left( -\sdfrac{2}{3}d - \sdfrac{1}{3} \right)
        \normu{\ba} \right)N^3 \biggr. \\
    &\phantom{(N+1)^d\biggl(} + \left( \sdfrac{1}{36}d^2 + \sdfrac{1}{180}d
        + \left( -\sdfrac{1}{3}d - \sdfrac{2}{3} + \normu{\ba} \right)\normu{\ba}\right.\\
        &\left.\phantom{\hspace*{.53\textwidth}}
        + \left( \sdfrac{2}{3}d + \sdfrac{1}{3} \right)\norm{\ba}^2\right)N^2 \\
    &\phantom{(N+1)^d\biggl(}\,\biggl. + \left( -\sdfrac{1}{30}d + \left(\sdfrac{1}{3}d
        + \sdfrac{2}{3} - 2\normu{\ba} \right)\norm{a}^2\right)N
        + \norm{\ba}^4 \biggr).
    \end{split}
\end{equation}
Finally, inserting the results from \eqref{eqsigmaaWTot} and \eqref{eqAW2} into \eqref{eqM2aW1}, we 
obtain the needed formula for $\fM_{2;\ba,\cW}(d,N)$.
\begin{lemma}\label{LemmaM2W}
Let $d,N \ge 1$ be integers, and let $A_{\ba,\cW}(d,N)$ be the average distance from a fixed point 
$\ba$ to the points in the hypercube $\cW$. 
Then, the second moment about $A_{\ba,\cW}(d,N)$~is
      \begin{equation*}
    \begin{split}
    \fM_{2;\ba,\cW}(d,N) &=\sdfrac{1}{(N+1)^d}\sum_{\bv\in\cW}
    \left(\distance^4(\bv,\ba)
          -2\distance^2(\bv,\ba)A_{\ba,\cW}+A_{\ba,\cW}^2\right)\\
    &= \sdfrac{4}{45}dN^4 + \left(\sdfrac{17}{90}d - \sdfrac{1}{3}\normu{\ba}\right)N^3 \\
        &\phantom{=}+ \left(\sdfrac{1}{180}d - \sdfrac{2}{3}\normu{\ba} + \sdfrac{1}{3}\norm{\ba}^2\right)N^2
            + \left(-\sdfrac{1}{30}d + \sdfrac{2}{3}\norm{a}^2\right)N.
    \end{split}
\end{equation*}

\end{lemma}

%%%%%%%%%%%%%%%%%%%%%%%%%%%%%%%%%%%%%%%%%%%%%%%%%%%%%%%%%%%%%%%%%%%%%%%%

\section{\texorpdfstring{The chance to find points in $\cW$ that are at some uncommon spacing from each other}{The chance to find points in W that are at some uncommon spacing from each other}}\label{SectionSpacingsW}
The formulas for the average and the second moment obtained in Section~\ref{SectionAMW}
allows us to estimate the chance to find points in $\cW$ that are situated at uncommon 
(away from the average) spacing from each other. 
It turns out that, as the dimension~$d$ gets larger, the probability to select at random two points
from $\cW$ that are closer or farther away from the average is smaller and smaller, reducing to zero as 
$d$ tends to infinity.

Following the same argument used in the proof of Theorem~\ref{TheoremV1}, we obtain the 
following result, which shows that for any fixed $\ba\in\RR^d$, almost
all normalized distances from $\ba$ to the points in $\cW$ are close to $\sqrt{A_{\ba,\cW}/dN^2}$.

\begin{theorem}\label{TheoremW1}
Let $\eta\in (0,1/2)$ be fixed. 
Let $B_{\ba,\cW}=A_{\ba,\cW}/dN^2$ denote the normalized average of the square distance from 
$\ba$ to points in $\cW$.
Then, for any integers $d\ge 2$, $N\ge 1$, and any point $\ba\in\RR^d$, we have 
\begin{equation*}%\label{eqPTV1}
       \sdfrac{1}{\#\cW} \#\left\{\bv\in \cW :
        \distance_d(\ba,\bv)\in \left[\sqrt{B_{\ba,\cW}}-\sdfrac{1}{d^\eta}, 
        \sqrt{B_{\ba,\cW}}+\sdfrac{1}{d^{\eta}} \right]
\right\}
\ge 1- \sdfrac{51}{15}\sdfrac{1}{d^{1-2\eta}}.
\end{equation*}
\end{theorem}

%%%%%%%%%%%%%%%%%%%%%%%%%
% \subsection{Triangles involving a fixed point}\label{SectionWTriangles}
We can continue our quest by looking for triplets of points in $\cW$.
In the same way as we proceeded in Section~\ref{SectionVTriangles}, we see that almost all pairs 
of distinct points \mbox{$(\bvi,\bvii)\in\cW^2$} have components situated at a distance close to 
$\sqrt{B_{\ba,\cW}}$ from~$\ba$. This means that almost all triangles formed by $\ba$ 
and two other points in the cube are `almost isosceles'. 
We can make the argument explicit, as we did in Theorem~\ref{ThVIsosceles}  for vertices, 
to find the following analogous result.

\begin{theorem}\label{ThWIsosceles}
Let $\cT_{\ba,\cW^2}\subset\cW^2$ be the set of all pairs of integer points $(\bvi,\bvii)$ which form a
non-degenerate triangle together with $\ba$. Let $\eta\in(0,1/2)$ be fixed. Then, for any integers
$d\geq2$, $N\geq1$, and any point $\ba\in\cW$, we have
\begin{equation*}
    \sdfrac{1}{\#\cT_{\ba,\cW^2}} \#\left\{ 
        (\bvi,\bvii)\in\cT_{\ba,\cW^2} \colon \abs{\distance_d(\ba,\bvi)-\distance_d(\ba,\bvii)}\leq\sdfrac{2}{d^{\eta}} \right\}
        \geq 1-\sdfrac{102}{15}\sdfrac{1}{d^{1-2\eta}}.
\end{equation*}
\end{theorem}

The sides of these triangles can be found explicitly.
To see this, we first use Lemma~\ref{LemmaAW} to express the normalized average solely 
on the distance from the center of the cube to $\ba$. Thus, using ~\eqref{eqDistanceCba}, we have
\begin{equation*}\label{eqWAverageSquares}
\begin{split}
    B_{\ba,\cW} &= \sdfrac{1}{3}+\sdfrac{1}{6N} - \frac{\normu{\ba}}{dN}+\frac{\norm{\ba}^2}{dN^2} \\
    &= \sdfrac{1}{12}+\sdfrac{1}{6N}+\left(\sdfrac{1}{4} - \frac{\normu{\ba}}{dN}
    +\frac{\norm{\ba}^2}{dN^2}\right) \\
    &= \sdfrac{1}{12}+\sdfrac{1}{6N} + \distance_d^2(\ba,\bc).
\end{split}
\end{equation*}
Here, employing Theorem~\ref{TheoremW1}, the first thing of which we make note, is that
for almost all points $\bv\in\cW$, the square distance $\distance_d^2(\bc,\bv)$ 
is close to $1/12+1/(6N)$.
It also follows that, for almost all pairs of points
$\bvi,\bvii\in\cW$, their mutual distance $\distance_d(\bvi,\bvii)$ 
is close to $\sqrt{B_{\bvi,\cW}}$, which is itself
close to $\sqrt{1/6+1/(3N)}$.
Therefore, with our earlier notation $r_{\ba}:=\distance_d(\ba,\bc)$, we find that almost all triangles $(\ba,\bvi,\bvii)$ have side lengths close to $\sqrt{1/6+1/(3N)}$,
$\sqrt{1/12+1/(6N)+r_{\ba}^2}$, and $\sqrt{1/12+1/(6N)+r_{\ba}^2}$.

If $r_{\ba}=0$, which occurs when $\ba$ is the center of the cube, we see that 
almost all triangles $(\ba,\bvi,\bvii)$
are almost right triangles. On the other hand, if $r_{\ba}$ is close to 
$\sqrt{1/12+1/(6N)}$, then almost all triangles
$(\ba,\bvi,\bvii)$ are almost equilateral.

In order to make this remarks explicit, we first use the analogue 
of~\eqref{eqThmV1SquaresBound} in the proof of Theorem~\ref{TheoremW1} to see that
\begin{equation*}
    \sdfrac{1}{\#\cW}  \#\left\{ \bvi\in\cW :
        \abs{\distance_d^2(\bc,\bvi)-\left( \sdfrac{1}{12}+\sdfrac{1}{6N} \right)} \geq \sdfrac{1}{2\sqrt{6}d^{\eta}} \right\}
        \leq \sdfrac{102}{15}\sdfrac{1}{d^{1-2\eta}}.
\end{equation*}
Furthermore, if $\bvi$ is a fixed point such that
\begin{equation*}
    \distance_d^2(\bc,\bvi)\in\left[ \sdfrac{1}{12}+\sdfrac{1}{6N}-\sdfrac{1}{2\sqrt{6}d^{\eta}},\
        \sdfrac{1}{12}+\sdfrac{1}{6N}+\sdfrac{1}{2\sqrt{6}d^{\eta}} \right],
\end{equation*}
then,
\begin{equation*}
\begin{split}
    &\phantom{\leq}\sdfrac{1}{\#\cW} \#\left\{ \bvii\in\cW :
        \abs{\distance_d(\bvi,\bvii)-\sqrt{\sdfrac{1}{6}+\sdfrac{1}{3N}}} \geq \sdfrac{1}{d^{\eta}} \right\} \\
    &\leq \sdfrac{1}{\#\cW} \#\left\{ \bvii\in\cW :
        \abs{\distance_d^2(\bvi,\bvii)-\left( \sdfrac{1}{6}+\sdfrac{1}{3N} \right)} \geq \sdfrac{1}{\sqrt{6}d^{\eta}} \right\} \\
    &\leq \sdfrac{1}{\#\cW} \#\left\{ \bvii\in\cW : 
        \abs{\distance_d^2(\bvi,\bvii)-B_{\bvi,\cW}} \geq \sdfrac{1}{2\sqrt{6}d^{\eta}} \right\} \\
    &\leq \sdfrac{102}{15}\cdot\sdfrac{1}{d^{1-2\eta}}.
\end{split}
\end{equation*}
Then, we can argue just as we did in the proof of 
Theorem~\ref{ThVSimilarTriangles} to find the proportion of non-degenerate triangles
$(\ba,\bvi, \bvii)$ such that
\begin{equation}\label{eqCommonTriangleConditionsW}
\begin{split}
    \distance_d(\ba,\bvi), \distance_d(\ba,\bvii)  
    &\in \left[ \sqrt{B_{\ba,\cW}}-\sdfrac{1}{d^\eta},\ \sqrt{B_{\ba,\cW}}+\sdfrac{1}{d^\eta} \right], 
    \text{ and }\\
%     \distance_d(\ba,\bvii) &\in \left[ \sqrt{B_{\ba,\cW}}-\sdfrac{1}{d^\eta}, \sqrt{B_{\ba,\cW}}+\sdfrac{1}{d^\eta} \right], \\
    \distance_d(\bvi,\bvii) &\in \left[ \sqrt{\sdfrac{1}{6}+\sdfrac{1}{3N}}-\sdfrac{1}{d^\eta},\ \sqrt{\sdfrac{1}{6}+\sdfrac{1}{3N}}+\sdfrac{1}{d^\eta} \right],
\end{split}
\end{equation}
and arrive at the following result.

\begin{theorem}\label{ThWSimilarTriangles}
Let $\cT_{\ba,\cW^2}\subset\cW^2$ be the set of all pairs of integer points $(\bvi,\bvii)$, which
together with $\ba$, form a non-degenerate triangle.
 Fix $\eta\in(0,1/2)$. Then, for any integers $d\geq2$, $N\geq1$, and any point $\ba\in\cW$, we have
    \begin{equation*}
        \sdfrac{1}{\#\cT_{\ba,\cW^2}} \#\left\{ (\bvi,\bvii)\in\cT_{\ba,\cW^2} :
            \ba, \bvi, \bvii \text{ satisfy } \eqref{eqCommonTriangleConditionsW} \right\}
            \geq 1- \sdfrac{102}{5}\sdfrac{1}{d^{1-2\eta}}.
    \end{equation*}
\end{theorem}

%%%%%%%%%%%%%%%%%%%%%%%%%%%%%%%%%%%%%%%%%%%%%%%%%%%%%%%%%%%%%%%%%%%%%%

\end{document}